 \documentclass[11pt]{amsart}
 \usepackage{amssymb}
 \usepackage{amssymb}
\usepackage{color}   
 
\begin{document}
 	\newcommand{\dbar}{\ensuremath{\overline\partial}}
 	\newcommand{\dbarstar}{\ensuremath{\overline\partial^*}}
 	\newcommand{\de}{\ensuremath{\partial}}
 	\newcommand{\C}{\ensuremath{\mathbb{C}}}
 	\newcommand{\R}{\ensuremath{\mathbb{R}}}
 	\newcommand{\D}{\ensuremath{\mathbb{D}}}
 	\newcommand{\T}{\ensuremath{\mathbb{T}}}
 	
 	\makeatletter
 	\newcommand{\sumprime}{\if@display\sideset{}{'}\sum%
 		\else\sum'\fi}
 	\makeatother
 	
 	\numberwithin{equation}{section}
 	
 	\newtheorem{theorem}{Theorem}[section]
 	\newtheorem{proposition}[theorem]{Proposition}
 	\newtheorem{conjecture}[theorem]{Conjecture}
 	\def\theconjecture{\unskip}
 	\newtheorem{corollary}[theorem]{Corollary}
 	\newtheorem{lemma}[theorem]{Lemma}
 	\newtheorem{observation}[theorem]{Observation}
 	\theoremstyle{definition}
 	\newtheorem{definition}{Definition}
 	\numberwithin{definition}{section}
 	\newtheorem{remark}{Remark}
 	\def\theremark{\unskip}
 	\newtheorem{question}{Question}
 	\def\thequestion{\unskip}
 	\newtheorem{example}{Example}
 	\def\theexample{\unskip}
 	\newtheorem{problem}{Problem}
 	
 	\def\vvv{\ensuremath{\mid\!\mid\!\mid}}
 	\def\intprod{\mathbin{\lr54}}
 	\def\reals{{\mathbb R}}
 	\def\integers{{\mathbb Z}}
 	\def\N{{\mathbb N}}
 	\def\complex{{\mathbb C}\/}
 	\def\CP{{\mathbb{CP}}\/}
 	\def\P{{\mathbb P}\/}
 	\def\dist{\operatorname{dist}\,}
 	\def\spec{\operatorname{spec}\,}
 	\def\interior{\operatorname{int}\,}
 	\def\trace{\operatorname{tr}\,}
 	\def\fs{\operatorname{\small FS}\,}
 	\def\I{\operatorname{I}}
 	\def\cl{\operatorname{cl}\,}
 	\def\essspec{\operatorname{esspec}\,}
 	\def\range{\operatorname{\mathcal R}\,}
 	\def\kernel{\operatorname{\mathcal N}\,}
 	\def\dom{\operatorname{dom}\,}
 	\def\linearspan{\operatorname{span}\,}
 	\def\lip{\operatorname{Lip}\,}
 	\def\sgn{\operatorname{sgn}\,}
 	\def\Z{ {\mathbb Z} }
 	\def\eps{\varepsilon}
 	\def\p{\partial}
 	\def\rp{{ ^{-1} }}
 	\def\Re{\operatorname{Re\,} }
 	\def\Im{\operatorname{Im\,} }
 	\def\dbarb{\bar\partial_b}
 	\def\eps{\varepsilon}
 	\def\Lip{\operatorname{Lip\,}}
 	\def\trace{\operatorname{tr}\,}
 	\def\df{\operatorname{DF}}
 	\def\fs{\operatorname{FS}}
 	\def\ec{\operatorname{E}}

 	\def\Hs{{\mathcal H}}
 	\def\E{{\mathcal E}}
 	\def\scriptu{{\mathcal U}}
 	\def\scriptr{{\mathcal R}}
 	\def\scripta{{\mathcal A}}
 	\def\scriptc{{\mathcal C}}
 	\def\scriptd{{\mathcal D}}
 	\def\scripti{{\mathcal I}}
 	\def\scriptk{{\mathcal K}}
 	\def\scripth{{\mathcal H}}
 	\def\scriptm{{\mathcal M}}
 	\def\scriptn{{\mathcal N}}
 	\def\scripte{{\mathcal E}}
 	\def\scriptt{{\mathcal T}}
 	\def\scriptr{{\mathcal R}}
 	\def\scripts{{\mathcal S}}
 	\def\scriptb{{\mathcal B}}
 	\def\scriptf{{\mathcal F}}
 	\def\scriptg{{\mathcal G}}
 	\def\scriptl{{\mathcal L}}
 	\def\scripto{{\mathfrak o}}
 	\def\scriptv{{\mathcal V}}
 	\def\frakg{{\mathfrak g}}
 	\def\frakG{{\mathfrak G}}
 	
 	\def\ov{\overline}
	\newcommand{\hb}{\mathbb H}
	\newcommand{\cx}{\mathbb C}
 	\newcommand{\ol}{\overline}
 	
 	\author{Mei-Chi Shaw}
 	
 	\thanks
 	{The author  is supported in part by NSF grant. She would also like to thank   Professor Christine Laurent-Thi\'ebaut  and the referee for helpful comments.}

 	\address{Department of Mathematics, University of Notre Dame, Notre Dame, IN 46556}
 	\email{Mei-Chi.Shaw.1@nd.edu}

 	\title[]  
 	{   $L^2$-Sobolev Theory for $\dbar$ on  Domains  in $\CP^n$}

 	\begin{abstract} In this article, we study  the range  of  the Cauchy-Riemann  operator $\dbar$   on    domains in the complex projective space $\CP^n$.  In particular, we show that $\dbar$ does not have closed range in $L^2$  for   (2,1)-forms  on the Hartogs triangle in $\CP^2$.  We also study the $\dbar$-Cauchy problem on pseudoconvex domains and use it to prove the Sobolev estimates for  $\dbar$ on pseudoconcave domains in $\CP^n$. 
 	 
	\end{abstract}	\maketitle
 	\bigskip

 \centerline{\it Dedicated to the memory of  Professor Joseph J. Kohn}

 \bigskip

 	\noindent{{\sc Mathematics Subject Classification} (2010): 32W05, 35N15.}
 	
 	\smallskip
 	
 	\noindent{{\sc Keywords}: Cauchy-Riemann operator;  Hartogs triangle,  complex projective space}

\section{Introduction}

 Since the fundamental work of Kohn for the $\dbar$-Neumann problem on smooth bounded strongly pseudoconvex 
domains in $\C^n$,   there has been tremendous progress on $L^2$-Sobolev theory of the
$\dbar$-operator and the $\dbar$-Neumann problem for bounded pseudoconvex domains in $\C^n$.
In particular, Kohn proved the following  landmark  results (see  \cite{Kohn63,Kohn64}).  
 
 \begin{theorem}[\textbf{Kohn 1963}] \label{th:Kohn63}Let $\Omega$ be bounded  strongly pseudoconvex domain with  smooth  boundary  in   $\C^n$, $n\ge 2$.  
Then the $\dbar$-Neumann operator 
$$N:L^2_{p,q}(\Omega)\to L^2_{p,q}(\Omega)$$ exists on $\Omega$. 
Furthermore,         the following sub-elliptic estimates hold for any $s\ge 0$
 $$\|Nf\|_{s+1}\le C \|f\|_s$$   
 and  $$\|\dbar^* Nf\|_{s+\frac 12}\le C \|f\|_s$$  
where $\|\ \|_s$ denotes  the Sobolev $s$-norm $W^s(\Omega)$. \end{theorem}
The solution  $\dbar^*Nf$ is called the {\it canonical solution} (or Kohn's solution), since it is the energy minimizing solution.

\begin{corollary} 
  Let     $f\in C^\infty_{p,q}( \overline \Omega)$ with $\dbar  f=0$ in $\Omega$, where $0\le p\le n$ and $1\le q\le n$.  There exists  $u= \dbar^* N f \in  C^\infty_{p,q-1}(\overline \Omega)$ satisfying 
$\dbar u=f$ in $\Omega$.  \end{corollary}

Another  important result  for $\dbar$ is the global regularity for $\dbar$ proved later by Kohn based on the weighted $\dbar$-Neumann problem   (see \cite{Kohn73}).

 \begin{theorem}[\textbf{Kohn 1973}] \label{th:Kohn main}Let $\Omega$ be bounded pseudoconvex domain with  smooth  boundary  in   $\C^n$, $n\ge 2$.  
  Let     $f\in W^s_{p,q}( \Omega)$ with $\dbar  f=0$ in $\Omega$, where $0\le p\le n$ and $1\le q\le n$ and  $s> 0$,  there exists  $u_s\in  W^s_{p,q-1}(\Omega)$ satisfying 
$\dbar u_s=f$ in $\Omega$.  \end{theorem}

  \begin{corollary}\label{co:global Kohn} Let $\Omega$ be bounded pseudoconvex domain with  smooth boundary  in   $\C^n$, $n\ge 2$.  
  Let     $f\in C^\infty_{p,q}(\overline \Omega)$ with $\dbar  f=0$ in $\Omega$, where $0\le p\le n$ and $1\le q\le n$. Then   there exists $u\in C^\infty_{p,q-1}(\overline \Omega)$ satisfying 
$\dbar u=f$ in $\Omega$.  \end{corollary}

  For $s\ge 0$,  let  $H^{p,q}_{W^s}(\Omega)$ be the Dolbeault cohomology with Sobolev $W^s$ coefficients   defined   by 
$$H^{p,q}_{W^s}(\Omega)=\frac { \{f\in W^s_{p,q}(\Omega)\mid \dbar f=0\}}{\{f\in W^s_{p,q}(\Omega) \mid 
f=\dbar u, u\in W^{s}_{p,q-1}(\Omega)\}}.$$
When $s=0$, we also use the notation $H^{p,q}_{L^2}(\Omega)$ for the $L^2$ Dolbeault cohomology.

Similarly, we use $H^{p,q}(\Omega)$ and $H^{p,q}(\overline \Omega)$ to denote the Dolbeault cohiomology group for $(p,q)$-forms  with $C^\infty(\Omega)$ and $C^\infty(\overline \Omega)$ coefficients respectively. Using these notation, Theorem \ref{th:Kohn main}  can be formulated  as 
\begin{equation}\label{eq:Ws} H^{p,q}_{W^s}(\Omega)=0,\qquad s>0.\end{equation}
  Corollary \ref{co:global Kohn} can be written as  
\begin{equation}\label{eq:C infinity} H^{p,q}(\overline \Omega)=0\end{equation}

When $s=0$,       $L^2$ existence for $\dbar$  was proved by  H\"ormander \cite{Hormander65}  for bounded    pseudoconvex convex domains, not necessarily with smooth boundary.
   
\begin{theorem}[\textbf{H\"ormander 65}]\label{th:L2 Hormander 0} Let   $\Omega$ be a bounded pseudoconvex domain  in $\C^n$. Then   $$H^{p,q}_{L^2}(\Omega)=0,\qquad q>0.$$
\end{theorem}

 Similar  results also hold for domains in a Stein manifold.  Both  $L^2$ and Sobolev  regularity for $\dbar$   have numerous applications.  Though the $L^2$ and Sobolev theory for $\dbar$ has been studied extensively  for domains in $\C^n$,  (see 
monographs \cite{FollandKohn, Hormander90, ChenShaw01, Straube10} for expositions on the subject), much less is known for $\dbar$ on domains 
in complex manifolds which are not Stein.  

 In this paper we   present some recent  results  of the 
$L^2$ and Sobolev theory for $\dbar$ on domains  in the complex projective space $\CP^n$.  
  There are many known results on $L^2$ existence theorems for $\dbar$ on 
pseudoconvex domains in $\CP^n$. In particular, we have 
$$H^{p,q}_{L^2}(\Omega)=0 \qquad \text{for all } 0\le p\le n,\ 1\le q\le n$$ 
for any pseudoconvex domain $\Omega\subset \CP^n$  with Lipschitz boundary $b\Omega$ (see Theorem \ref{th:L2 Lipschitz}).

 One of the main results in this paper  is to show that $\dbar$ might not have closed range on some    pseudoconvex domain in $\CP^2$  if the Lipschitz condition is dropped.  In particular, Theorem \ref{th:L2 Hormander 0} does not hold for arbitrary pseudoconvex  domains in $\CP^n$.  
 The examples are given by  the Hartogs triangles in $\CP^2$ (see   Theorem \ref{th:Nonclosed 21} and Corollary \ref{co:non Hausdorff}).
  
   The Hartogs triangles in $\C^2$ and $\CP^2$  are important examples of domains which are not   Lipschitz. Hartogs triangles in $\CP^2$  are also interesting examples  in complex foliation theory. They can be viewed as  Levi-flat hypersurfaces with singularities  since they are both pseudoconvex and pseudoconcave.
Non-closed range properties for $\dbar$ on domains in complex manifolds with smooth Levi-flat boundaries have been obtained in \cite{ChakrabartiShaw15} (see also \cite{LaurentShaw15}).

The plan of the paper is as follows.
  In section~\ref{sec:FS}  we summarize  some known results for $\dbar$ in $L^2$  for pseudoconvex domains in $\CP^n$ and  give    an alternative proof for Theorem \ref{th:L2 Hormander 0}  (see Theorem \ref{th:L2 Hormander}).   In section \ref{sec:concave}, we prove the Sobolev estimates for $\dbar$ on pseudoconcave domains in $\CP^n$ using the $\dbar$-Cauchy problem with weights. 
In section \ref{sec:Hartogs 21}, we give some basic properties of holomorphic  functions and forms on the Hartogs triangles. The non-closed range property for $\dbar$ for $(2,1)$-forms with $L^2$ coefficients  on  Hartogs triangles in $\CP^2$ is proved  in section \ref{sec:nonclosed proof}. 
  Theorem \ref{th:Nonclosed 21}   is a stronger assertion of the earlier results proved in  \cite{LaurentShaw18} and \cite{BFLS22}.  The dimension of $H^{2,1}_{L^2}(\hb^+)$ is 
 not only infinite, but is uncountable.

We collect  several  open problems which 
  are related to the $L^2$ or  Sobolev estimates for $\dbar$ on  domains in $\C^n$ or  $\CP^n$  in section \ref{sec:Open Problems}.
  Since  Hartogs triangles in $\CP^2$  are both pseudoconvex and pseudconcave, Theorem \ref{th:Nonclosed 21}  also provide  examples that  $\dbar$ does not have closed $L^2$  range  on some   {\it  pseudoconcave} domains  in $\CP^2$. We remark that $L^2$ theory for even  
 pseudoconcave domains with smooth boundary  remains an open question (see Problem \ref{prob:L2 pseudoconcave}). The missing ingredient is exactly 
 the lack of  Kohn's type     Sobolev estimates \eqref{eq:Ws} for  pseudoconvex domains in $\CP^n$ (see Problem \ref{prob:Sobolev CPn}).

  \section{$L^2$ theory for $\dbar$  on pseudoconvex domains  in $\CP^n$}\label{sec:FS}
 In this section, we review some known results on pseudoconvex domains in $\CP^n$. 
Let $\omega$ be the K\"{a}hler form associated with the Fubini-Study metric in $\CP^n$.  
Let $\Omega$ be a pseudoconvex  domain  in $\CP^n$ such that $\overline \Omega\neq \CP^n$ and $\Omega$ has   $C^2$-smooth boundary $b\Omega$.  Let $\delta$ be the distance function from $z$ to $b\Omega$.    
 
 Let 
$\delta(z)=dist(z, b\Omega )$ be the distance, with respect to the Fubini-Study metric, from $z$ to the boundary $b\Omega$. Let $\Omega_\eps=\{z\in\Omega \mid \delta(z)>\eps\}$. It then follows from Takeuchi's theorem~\cite{Takeuchi64} that there exists a universal constant $K_0>0$ such that
\begin{equation}\label{eq:Takeuchi}   i\partial \dbar  (-\log \delta ) \ge K_0\omega
\end{equation}
on $\Omega$. In particular, there exists $\epsilon_0>0$ such that    
\begin{equation}\label{eq:2.1}
\partial \dbar  (-\delta)(\zeta,\overline \zeta)\ge K_0\epsilon |\zeta|^2_{\omega}
\end{equation} for all
$\zeta\in T^{1,0}_x(b\Omega_\epsilon)$ for $ 0\le \epsilon \le \epsilon_0 $.
(See also \cite{GreeneWu78, CaoShaw05} for  different proofs of Takeuchi's theorem.)   
  
  Using \eqref{eq:2.1}, we have  the following theorem.
  
  \begin{theorem}[\textbf{Takeuchi}]\label{th:Takeuchi} Let $\Omega$ be a pseudoconvex domain in $\CP^n$ such that $\overline \Omega\neq \CP^n$. Then 
  the Dolbeault cohomology group $H^{p,q}(\Omega)=0$ for all $q>0$. 
  \end{theorem}

  \subsection{$L^2$ existence theorems for $(0,q)$-forms}

    Let $L^2_{p, q}(\Omega)$ be the space of $(p, q)$-forms $u$ on $\Omega$ with respect to the Fubini-Study metric $\omega$ such that
 \[
 \|u\|^2_\omega=\int_{\Omega} |u|_\omega^2  dV_\omega <\infty.
 \]
 We will also use $(\cdot, \ \cdot)_\omega$ to denote the associated inner product.  Let   $\dbar\colon L^2_{p, q-1}(\Omega)\to L^2_{p, q}(\Omega)$ be the weak maximal closure of $\dbar$   and 
 let $\dbar^*_\omega$ be the Hilbert space  adjoint of $\dbar$. We now recall an integration by parts formula.

  Let $L_1, \ldots, L_n$ be a local orthonormal frame field of type $(1, 0)$  and $\phi^1, \ldots, \phi^n$ be the coframe field.   For a $(p, q)$-form $u$, we set
 \[
 \langle \Theta u, u\rangle_\omega =\sum_{j, k=1}^n \langle \bar{\phi}^j\wedge \big( \bar L_k\lrcorner R(L_j, \bar L_k) u\big), \ u\rangle_\omega,
 \]
 where $R$ is the curvature  operator on $(p,q)$-forms with respect to the Fubini-Study metric and  $\lrcorner$ is the usual contraction operator 
 (see e.g \cite{CSW04, FuShaw22}).  
We  have that if $u$ is a $(p, q)$-form on $\CP^n$ with $q\ge 1$, then
\begin{equation}\label{eq:c1}
\langle\Theta u, u\rangle_\omega=0, \quad \text{if }\quad p=n; \quad\quad 
\langle\Theta u, u\rangle_\omega\ge 0, \quad  \text{if}\quad p\ge 1;
\end{equation}
and
\begin{equation}\label{eq:c2}
\langle\Theta u, u\rangle_\omega=q(2n+1)|u|^2 \quad \text{if}\quad  p=0. 	
\end{equation}
 With the above notations, we can now state the following 
 {\it Basic Identity} (see \cite{Wu88, Siu00, CSW04}).
 
 \begin{theorem}[\textbf{Bochner-Kodaira-Morrey-Kohn}]  Let $\Omega$ be a domain  with  $C^2$  boundary $b\Omega$ in in $\CP^n$ and let $\omega$ be the Fubini-Study metric. 
  For any $u\in C^1_{p, q}(\overline{\Omega})\cap \dom(\dbar^*_\omega)$, we have
 \begin{equation}\label{eq:BKMKH}
 	\|\dbar u\|^2_\omega+\|\dbar^*_\omega u\|_{\omega}^2 =\|\overline\nabla u\|_{\omega}^2 +(\Theta u,\ u)_\omega+  
 	\int_{b\Omega} \langle(\partial \dbar \rho) u, \ u\rangle_\omega  dS_\omega 
 \end{equation}
 where  $dS_\omega$ is the induced surface element on $b\Omega$ and 
$
 \|\overline\nabla u\|_\omega^2=\sum_{j=1}^n \|\nabla_{\bar L_j} u\|_\omega^2.  
$
 \end{theorem}

For a proof of these results, see  \cite{Wu88} or  Proposition A.5 in    \cite{CSW04}. We first exploit the positivity of the curvature $\Theta$  for $(0,q)$-forms.  When $p=0$, we have the following proposition using \eqref{eq:c2}. 

\begin{proposition}\label{prop:dbar1} Let $\Omega$ be a 
pseudoconvex domain in $\CP^n$ with $C^2$ boundary and $1\le q\le n-1$.  Then 
	\begin{equation}\label{eq:e1}
		\|\dbar u\|_\omega^2+\|\dbarstar u\|_\omega^2\ge q(2n+1)\|u\|_\omega^2
	\end{equation} 
for any $(0, q)$-form $u\in \dom(\dbar)\cap\dom(\dbarstar_\omega)$. \end{proposition}

\begin{theorem}[\textbf{$L^2$ existence for $(0,q)$-Forms}] \label{th:L2 cpn} Let $\Omega$ be a 
pseudoconvex domain in $\CP^n$ such that $\overline \Omega \neq \CP^n$   and $1\le q\le n-1$.      For any $\dbar$-closed $(0,q)$-form
 $f\in L^2_{0, q}(\Omega)$,
there exists a $(0, q-1)$-form
$u\in L^2_{0, q-1}(\Omega)$ such that $\dbar u= f$ with 
\begin{equation}\label{eq:e2}
	  \|u\|_\omega^2\le \frac{1}{q(2n+1)} \|f\|_\omega^2. 
\end{equation}

\end{theorem}
\begin{proof} If $\Omega$ has $C^2$ boundary,   estimate \eqref{eq:e2} is then a consequence of \eqref{eq:e1}. 
 The general case is then proved by exhausting $\Omega$ from inside by pseudoconvex domains with smooth boundaries. 
\end{proof}

Notice that the constant   ${1}/{q(2n+1)}$ in \eqref{eq:e2} is independent of the diameter of the domain in $\CP^n$ with respect to the Fubini-Study metric. 

\begin{corollary}\label{co:L2 vanishing}  Let $\Omega$ be a 
pseudoconvex domain in $\CP^n$ such that $\overline \Omega \neq \CP^n$.  We have
$$H^{0,q}_{L^2} (\Omega)=0 \qquad \text{for every }q>0.$$
\end{corollary}

 Theorem \ref{th:L2 cpn} gives  an alternative proof of  H\"ormander's $L^2$ existence for $\dbar$ for bounded pseudoconvex domains in $\C^n$. 

\begin{theorem}[\textbf{H\"ormander's Theorem Revisited}]\label{th:L2 Hormander} Let $\Omega$ be a bounded  pseudoconvex domain in $\C^n$  with diameter $d$, where $d=\sup_{z,z'\in \Omega} |z-z'|$.   Then for any $f\in L^2_{p, q}(\Omega)$ with $\dbar f=0$, there is a $(p, q-1)$-form $u\in L^2_{(p, q-1)}(\Omega)$ such that $\dbar u= f$ with
\begin{equation}\label{eq:L2 bound}
	  \|u\|^2 \le  c_{n,q} d^2 \|f\|^2
\end{equation}	 
where $\|\ \|$ is the Euclidean norm and  
\begin{equation}\label{eq:c n q} c_{n,q}=   \frac {e(n+2q+1)}{4q(2n+1)}.\end{equation}  
\end{theorem}

\begin{proof} Since the domain $\Omega$ is in $\C^n$, $p$ plays no role. We  may assume that $p=0$. We first embed $\C^n$ in $\CP^n$ and  view $\Omega$ as a domain in $\CP^n$ endowed with the Fubini-Study metric.  Using Theorem \ref{th:L2 cpn}, there exists $u\in L^2_{0,q-1}(\Omega)$ such that 
\begin{equation}\label{eq:u fs} \|u\|_\omega^2\le \frac{1}{q(2n+1)} \|f\|_\omega^2. 
\end{equation}

Next we compare norm $\|\ \|_\omega$ with the Euclidean norm $\|\ \|$.  
 The K\"{a}hler form $\omega$  is given by 
 \begin{align}
 \omega&= i\partial\dbar \log (1+|z|^2) \\
             &=i\sum_{\alpha, \beta=1}^n g_{\alpha\bar\beta}(z)\, dz_\alpha\wedge d\bar z_\beta
 \end{align}
 where
 \begin{equation}\label{eq:g1}
 g_{\alpha\bar\beta}(z)=\frac{\partial^2\log (1+|z|^2)} {\partial z_\alpha \partial \bar z_\beta} =\frac{(1+|z|^2)\delta_{\alpha\bar\beta}-\bar z_\alpha z_\beta}{(1+|z|^2)^2}.
 \end{equation}
The volume form $dV_\omega$ with respect to $\omega$  is   
\begin{equation}\label{eq:g2}
dV_{\omega}=\det(g_{\alpha\bar\beta}(z)) dV_{\ec}=\frac{1}{(1+|z|^2)^{n+1}} dV_{\ec}
\end{equation}
where $dV_{\ec}$ is the Euclidean volume form. Furthermore, we have
\begin{equation}\label{eq:w we}\frac 1{(1+|z|^2)^2}\omega_{\ec} \le \omega\le\frac 1{1+|z|^2} \omega_{\ec}.\end{equation}

Assume that $\Omega$  has diameter  $d=2\epsilon$ in $\C^n$ with respect to the Euclidean norm for some $\epsilon>0$.    Without loss of generality, we may assume that  $\Omega\subset B_\epsilon(0)=\{z\in \C^n\mid |z|<\epsilon\}$.  Then using \eqref{eq:g2}, we have
$$  \frac 1 {(1+\epsilon^2)^{n+1}} dV_{\ec}\le dV_\omega\le   dV_{\ec}.$$   
Using  \eqref{eq:w we},   we have 
$$\frac 1{(1+\epsilon^2)^2}\omega_{\ec}\le \omega\le \omega_{\ec}$$ 
and 
\begin{equation}\label{eq:dz w}1\le |dz_j|_\omega \le (1+|z|^2).\end{equation}

  Let  $f=\sum_K f_K d\bar z_K$,  where $K$ is a multi-index with $|K|=q$. It follows from \eqref{eq:g2} and  \eqref{eq:dz w}  that 
$$ |f|^2_\omega dV_\omega \le |f|^2  (1+|z|^2)^{2q}  dV_\omega  \le (1+\epsilon^2)^{2q} |f|^2 dV_0,$$
 and 
\begin{equation}\label{eq:f e w} \|f\|^2_\omega\le (1+\epsilon^2)^{2q} \|f\|^2.\end{equation}
  Similarly, we have  
$$|u|^2_\omega  dV_\omega\ge |u|^2 dV_\omega \ge\frac 1{ (1+\epsilon^2)^{n+1}} |u|^2 dV_0,$$
and 
 \begin{equation}\label{eq:u e w} \|u\|^2_\omega \ge \frac 1{(1+\epsilon^2)^{n+1}} \|u\|^2.\end{equation}

It follows from \eqref{eq:u fs}, \eqref{eq:f e w}  and \eqref{eq:u e w}  that  we have
\begin{equation}\label{eq:u f epsilon} \| u\|^2\le \frac {(1+\epsilon^2)^{n+1+2q} }{q(2n+1)}\|f\|^2,\end{equation}
when we assume that  $\Omega$  has diameter $2\epsilon$.

Suppose that $\Omega$   lies in $B_1(0)$,  the ball of radius $1$. By  scaling, we have for any $\epsilon>0$, 
\begin{equation}\label{eq:u f epsilon} \| u\|^2\le \frac 1{q(2n+1)}\frac  {(1+\epsilon^2)^{n+1+2q} }{\epsilon^2}  \|f\|^2.\end{equation}
Since $\epsilon>0$ is arbitrary, we see that the function 
$$\phi(\epsilon) =\frac  {(1+\epsilon^2)^{n+1+2q} }{\epsilon^2} $$
achieves its minimum $m$  at $\epsilon = 1/\sqrt{n+2q}$ with 
$$m=  (n+2q+1) \bigg(1+\frac 1{n+2q}\bigg)^{n+2q}.$$ 
Since 
 $ (1+  1/(n+2q))^{n+2q}\nearrow e$, let    $c_{n,q}$ be defined by \eqref{eq:c n q}. Then 
 it follows from   \eqref{eq:u f epsilon}  that 
$$\|u\|^2\le   c_{n,q} 2^2 \|f\|^2.$$
 when $\Omega$  has diameter 2. 

Suppose the domain $\Omega$ in $\C^n$  is with arbitrary diameter  $d$. Notice that the Euclidean metric admits a dilation. Thus    \eqref{eq:L2 bound} follows easily from a scaling argument.
\end{proof}

\begin{remark}   Theorem \ref{th:L2 Hormander} is an alternative proof of the H\"ormander's $L^2$ theory (see Theorem \ref{th:L2 Hormander 0} and  \cite{Hormander65}). H\"ormander's method is to use the weight function $\varphi = t|z|^2$ to obtain  the $L^2$ existence  with estimate \eqref{eq:L2 bound} with   $c_{n,q}=e/q$.   In comparison, the proof here uses the positive curvature of the Fubini-Study metric and the constant $c_{n,q}$ given by \eqref{eq:c n q}  is comparable up to a factor  to H\"ormander's results. 

\end{remark}

\subsection{The $\dbar$-Neumann operator with weights}
When $p>0$, $L^2$ theory for $\dbar$ on a domain $\Omega$  in $\CP^n$  requires more work since the curvature $\Theta$ is only nonnegative.  
Let $\Omega$ be a  pseudoconvex  domain with Lipschitz boundary $b\Omega$ in  $\CP^n$.  We may assume that there exists a Lipschitz defining function $\rho=-\delta$ such that 
\begin{equation}\label{eq:delta weight}i\partial \dbar (-\log  \delta)\ge C\omega\end{equation}
for some $C>0$.  Let $t>0$ and let $\phi_t= -t \log \delta$. Then $\phi_t$ is a strictly plurisubharmonic function on $\Omega$. Using $\phi_t$ as the weight function in  H\"ormander's $L^2$ methods with the weight function $\phi_t$,   we have  
$$e^{-\phi_t}= e^{t\log\delta}=\delta^t.$$ 
 We use $\|\ \|_t$ to denote the $L^2$  norm with weight under the  under the Fubini-Study metric. 

Let  
$$\dbar :L^2_{p,q-1}({\Omega,\delta^t})\to L^2_{p,q}(\Omega,\delta^t)$$     be the  weak  maximal  $L^2$ closure of $\dbar$  and its  Hilbert space adjoint is denoted by  $\dbar^*_t$ 
 such that  $\dbar^*_t: L^2_{p,q}({\Omega,\delta^t})\to L^2_{p,q-1}(\Omega,\delta^t)$ and if $v\in \text{Dom}(\dbar^*_t)$ if and only if 
 $$(\dbar u, v)_t= (u, \dbar^*_t v)_t\qquad \text{for every } u\in \text{Dom}(\dbar).$$ 
Now we use the H\"ormander's $L^2$ methods with the weight function $\phi_t$  combined with the Bochner-Kodaira-Morrey-Kohn formula. Using the same notation as before, but now we suppress the dependence of the Fubini-Study metric $\omega$ and emphasize on the weighted norm with respect to $\phi_t$.

 \begin{theorem}[\textbf{Bochner-Kodaira-Morrey-Kohn-H\"ormander}]  Let $\Omega$ be a domain in $\CP^n$ with  $C^2$  boundary $b\Omega$.
  For any $u\in C^1_{p, q}(\overline{\Omega})\cap \dom(\dbar^*_t)$, we have
 \begin{equation}\label{eq:BKMKH1}
 	\|\dbar u\|^2_t+\|\dbar^*_t u\|_t^2 =\|\overline\nabla u\|_t^2 +(\Theta u,\ u)_t+ ((i\partial\dbar \phi_t) u, u)_t+ 
 	\int_{b\Omega} \langle(\partial \dbar \rho) u, \ u\rangle_t e^{-\phi_t} dS 
 \end{equation}
 where  $dS$ is the induced surface element on $b\Omega$ and 
$
 \|\overline\nabla u\|_t^2=\sum_{j=1}^n \|\nabla_{\bar L_j} u\|_t^2.  
$
 \end{theorem}
 
 \begin{corollary}[\textbf{Weighted $L^2$ Existence for $\dbar$}]\label{co:L2 weights} Let $\Omega$ be a domain in $\CP^n$ with   Lipschitz   boundary $b\Omega$. 
 For any $0\le  p \le n-1$, $1\le q\le n$  and   
 $f\in L^2_{p, q}(\Omega,\delta^t)$ such that $f$ is $\dbar$-closed,
there exists a $(p, q-1)$-form
$u\in L^2_{p, q-1}(\Omega,\delta^t)$ such that $\dbar u= f$ with 
\begin{equation} 
	C t \|u\|_t^2\le    \|f\|_t^2 
\end{equation}
where $C$ is the same constant as in \eqref{eq:delta weight}.

\end{corollary}
\begin{proof} Suppose $\Omega$ is a domain with $C^2$ boundary, this  follows from \eqref{eq:BKMKH1} since
$$\|\dbar u\|^2_t+\|\dbar^*_t u\|_t^2 \ge Ct \|u\|_t^2.$$
For a domain with Lipschitz boundary, we can use an exhaustion argument.
\end{proof}

Suppose that $\Omega$ is a bounded domain in $\C^n$ with $C^2$   boundary. Diederich and Fornaess \cite{DiederichFornaess77}  proved that there exists a defining function $\rho$ and an exponent $0<\eta<1$ such that $-(-\rho)^\eta$ is strictly plurisubharmonic in $\Omega$. 
Based on Takeuchi's theorem, 
Ohsawa and Sibony \cite{OhsawaSibony98} generalized the Diederich-Fornaess results to domains in $\CP^n$. In fact, they showed  that   one can take $\rho=-\delta$ where $\delta$ is the distance function from $z\in \Omega$ to $b\Omega$.   Ohsawa-Sibony results have been extended to domains with Lipschitz boundary by    Harrington  in \cite{Harrington17}, where he proved that there exists a Lipschitz defining function $\rho$ and $0<\eta<1$ such that 
\begin{equation}\label{eq:os}
i\partial \dbar (-\delta^{\eta})\ge  C\eta\delta^{\eta}\omega
\end{equation}
in the sense of currents  for some constant $C>0$.

  Using  \eqref{eq:os} and  a technique of Berndtsson-Charpentier \cite{BerndtssonCharpentier00}   have the following $L^2$ existence  theorem without weights. 

\begin{theorem}[\textbf{$L^2$ Existence for $(p,q)$-Forms}] \label{th:L2 Lipschitz}  Let $\Omega$ be a pseudoconvex domain in $\CP^n$ with Lipschitz boundary.  Then 
\begin{equation} \label{eq:L2 pq} H^{p,q}_{L^2}(\Omega)=0\qquad \text{for every }q>0.\end{equation}
Furthermore, for any $s<\frac \eta 2$, where  $\eta$  is the exponent in \eqref{eq:os}, we have 
\begin{equation} \label{eq:Ws pq} H^{p,q}_{W^s}(\Omega)=0\qquad \text{for every }q>0.\end{equation}
\end{theorem}

\begin{proof} 
We refer the reader to   \cite{BerndtssonCharpentier00,HenkinIordan00,FuShaw22} for a proof of this theorem. \end{proof}

The following proposition is a consequence of the above $L^2$-theory for $\dbar$ on $\CP^n$. Its proof
follows the same lines of arguments as those in \cite{GreeneWu79, Demailly82, HenkinIordan00, CSW04} when the 
boundary is $C^2$-smooth. 

\begin{proposition}\label{prop:L^2holo}  Let $\Omega$ be a pseudoconvex domain in $\CP^n$ with Lipschitz boundary.  Then  
the $L^2$ holomorphic $(n,0)$-forms in $L^2_{n,0}(\Omega)\neq \{0\}$. Furthermore, $L^2$  Holomorphic   $(n,0)$-forms
separate points. 
\end{proposition}

  Notice that  in Theorem \ref{th:L2 Lipschitz}, we need the assumption of Lipschitz boundary for $\Omega$ when $p>0$. We will prove that Theorem \ref{th:L2 Lipschitz} does not hold if the Lipschitz condition is dropped.  In contrast, if $p=0$,  we do not need any regularity for $\Omega$ in  Theorem \ref{th:L2 cpn}.

\begin{remark}  Bounded plurisubharmonic exhaustions functions have been studied in $\C^n$ and $\CP^n$ extensively. 
   The Diederich-Fornaess theorem  has been extended to pseudoconvex domains in $\C^n$  with Lipschitz boundary (see Demailly \cite{Demailly82}).  
  The Diederich-Fornaess exponent is the supremum of   $0<\eta<1$ such that \eqref{eq:os} holds.     It is related to  the  nonexistence of Levi-flat hypersurfaces in complex manifolds (see \cite{AB,FuShaw16, FuShaw18}).

 \subsection{The $\dbar$-Cauchy problem with weights} 
 For fixed  $t\ge 0$, let 
 $$\dbar_c:L^2_{p,q-1}({\Omega,\delta^{-t}})\to L^2_{p,q}({\Omega,\delta^{-t}})$$ be the minimal (strong) closure of $\dbar$. By this we mean that $f\in \text{Dom}(\dbar_c)$ if and only if  there exists a sequence of smooth compactly supported  forms  $f_\nu$ in $C^\infty_{p,n-1}(\Omega)$   such that  $f_\nu\to f$ and $\dbar f_\nu\to \dbar f$ in $L^2(\Omega, \delta^{-t})$.

\begin{lemma}\label{le:dbar star c} The following conditions are equivalent:
\begin{enumerate}
\item $\dbar: L^2_{p,q-1}(\Omega, \delta^t)\to L^2_{p,q}(\Omega, \delta^t)$ has closed range.
\item $\dbar^*_t: L^2_{p,q}(\Omega, \delta^t)\to L^2_{p,q-1}(\Omega, \delta^t)$ has closed range. 
\item $\dbar_c:L^2_{n-p,n-q}(\Omega,\delta^{-t})\to L^2_{n-p, n-q+1}(\Omega,\delta^{-t})$ has closed range.
\end{enumerate}
\end{lemma}
\begin{proof}
 It is well-known that  $\dbar$ has closed range if and only if $\dbar^*_t$ has closed range (see  \cite{Hormander65} or Lemma 4.1.1 in \cite{ChenShaw01}).
By using the Hodge star operator, we have that  (1) and (3) are equivalent 
   (see \cite{ChakrabartiShaw12,LaurentShaw13}). 
  \end{proof}

 \begin{theorem}[\textbf{$L^2$ Serre Duality with Weights}]\label{co:L^2vanishing} Let $\Omega$ be a pseudoconvex domain with Lipschitz boundary in  $\CP^n$.    We have 
for any $t\ge 0$, 
$$H^{p,q}_{L^2}(\Omega, \delta^t)\cong H^{n-p,n-q}_{\dbar_c,L^2}(\Omega, \delta^{-t})=\{0\},\qquad q\neq 0. $$
\end{theorem}

\begin{proof}  Using Corollary \ref{co:L2 weights},     $\dbar$ has closed      range in $L^2_{p,q}(\Omega,\delta^t)$ for all degrees and $t\ge 0$.   Thus using Lemma \ref{le:dbar star c} and  the $L^2$ Serre duality  (see \cite{ChakrabartiShaw12}), the theorem  follows.

   \end{proof}

 \begin{corollary}[\textbf{$\dbar$-Cauchy Problem in $L^2$ Spaces with Weights}]  \label{co:L^2Cauchy} Let $\Omega$
be a pseudoconvex domain with Lipschitz boundary in $\CP^n$, $n\ge 2$.          Suppose that 
$f\in L^2_{p,q}(\Omega,\delta^{-t})$ where $t\ge 0$,   $0\le p\le
n$ and
$1\le q< n $.
Assuming that 
$\dbar f=0$  in $\CP^n$ with $f=0$ outside $\Omega$.
Then   there exists  $u\in L^2_{p,q-1}(\Omega,\delta^{-t})$  with $u=0$ outside $\Omega$ satisfying
$\bar\partial u=f$ in the
distribution sense in
$\CP^n$.

For $q=n$,  if $f$ satisfies the compatibility condition 
\begin{equation}\label{eq:topdegree} 
\int_{\Omega} f\wedge \phi =0, \qquad \phi \in L^2_{n-p,0}(\Omega,\delta^t)\cap \text{Ker}(\dbar),
\end{equation}
then the same conclusion holds. 
\end{corollary}
\begin{proof} Since the boundary is Lipschtiz, we have that solving $\dbar_c$ is the same as solving $\dbar$ with prescribed support in $\overline \Omega$ (see Lemma 2.3 in \cite{LaurentShaw13}). 
   \end{proof}

  \end{remark}

  \section{Sobolev estimates for $\dbar$    on pseudoconcave domains in $\CP^n$}\label{sec:concave}


Let $\Omega$ be a pseudoconvex domain in $\CP^n$ with Lipschitz boundary, where $n\ge 2$.  We always assume that $\overline\Omega\neq \CP^n$.  Let  $\Omega^+$ be  the complement of $\overline \Omega$ defined by 
$$\Omega^+=\mathbb{CP}^n\setminus \overline\Omega.$$
Then  $\Omega^+$ is a pseudoconcave domain with Lipschitz boundary. 
  Estimates for the $\dbar$-equation in Sobolev spaces $W^k(\Omega^+)$ have been obtained  for $k=1$ in earlier papers using the $\dbar$-Cauchy problem 
   (see \cite{CSW04, CaoShaw07}. For  $k\ge 2$,  it is  proved in  Henkin-Iordan  \cite{HenkinIordan00}) under the condition that the boundary 
  $b\Omega^+$ is $C^2$. Here we will  give a streamlined proof of  Sobolev estimates for $k\ge 1$ for pseudoconcave  domains with Lipschtiz boundary using the $\dbar$-Cauchy problem with weights.

   Let $W^k_0(\Omega)$ be the Sobolev space which is the completion of $C^\infty_0(\Omega)$ under the $W^k(\Omega)$ norm. 
   We have the following characterization of the space  $W^k_0(\Omega)$.

\begin{lemma}\label{le:Wk compact} Let $\Omega$ be a bounded Lipschitz domain in $\CP^n$.  Let $\delta(z)$ be the distance function from $z\in \Omega$ to $b\Omega$. Then for $k\ge 1$, 
  $g\in W^k_0(\Omega)$ 
if and only if $g\in W^k(\Omega)$ and 
\begin{equation}\label{eq:delta 1}   \delta^{-s+|\beta|} D^{\beta}g \in {L^2(\Omega)}\qquad \text{for all }|\beta|<k.\end{equation}
\end{lemma}

 For a proof of this lemma, see    Theorem 11.8 in Lions-Magenes \cite{LionsMagenes72}, where theorem is stated for smooth domains.  Similar  proof can be applied to domains with Lipschtiz boundary  (see also Grisvard \cite{Grisvard85}).

\begin{theorem}\label{th:L^2Extension} Let $\Omega^+$ be a
pseudoconcave domain in $\CP^n$ with Lipschitz
boundary, $n\ge 2$. Let $k\in \N$. 
  For any $\dbar$-closed  $f\in W^{k}_{p,q}({\Omega}^+)$, where $0\le p\le n$, $0\le q< n-1$, there exists
$F\in W^{k-1}_{p,q}(\mathbb{CP}^n)$ with
$F|_{\Omega^+}=f$ and  $\dbar F=0$ in $\mathbb{CP}^n$ in the
distribution sense.

\end{theorem}
\begin{proof}  The  $W^1(\Omega^+)$ estimates have already been   proved earlier (see \cite{CSW04, CaoShaw07, FuShaw22,HenkinIordan00}).  We 
will show that the proof can be modified for $k\ge 2$. 
  Since $\Omega^+$ has Lipschitz  boundary, there exists a bounded extension
operator from
  $W^k(\Omega^+)$ to $W^k(\mathbb{CP}^n)$  (see, e.g.,  \cite{Stein70}).  Let
$\tilde f\in W^{k}_{p,q}(\mathbb{CP}^n)$  be the  extension of $f$  so that
$ \tilde f|_{ {\Omega}^+  } = f $
with
  $\|\tilde f\|_{W^{k}(\mathbb{CP}^n)}\le C\|f\|_{W^{k}(\Omega^+)}.$
We have  
$\dbar \tilde f \in L^2_{p,q+1}(\Omega, \delta^{-2k+2})$, where $\Omega= \CP^n\setminus \Omega^+)$. 

From Corollary \ref{co:L^2Cauchy},  there exists $u_c\in L^2_{p,q}(\Omega,\delta^{-2k+2})$ such that $\dbar u_c=\dbar \tilde f$. Extending $u_c$ to be zero outside $\overline \Omega$, we have   
$$\dbar u_c=\dbar \tilde f\qquad \text{in }\CP^n.$$ Since $u_c$ satisfies an elliptic system, we have that $u_c\in L^2_{p,q}(\Omega,\delta^{-2k+2})$  implies that  $u_c\in W^{k-1}_{0,p,q}(\Omega)$.   Define
\begin{equation}\label{eq:L2closed}F= \tilde f-u_c.\end{equation}
Then $F\in W^{k-1}_{p,q}(\mathbb{CP}^n)$ and $F$ is a $\dbar$-closed extension of $f$.
\end{proof}

\begin{corollary}\label{co:W^1holo}
 Let $\Omega^+$ be a pseudoconcave domain 
in $\CP^n$ with Lipschitz boundary, where $n\ge 2$.   Then
$W^{1}_{p,0}( \Omega^+)\cap \text{Ker}(\dbar)=\{0\}$  for every $1\le  p\le n$
and  $W^{1} ( \Omega^+)\cap \text{Ker}(\dbar)=\C$.
\end{corollary}
 \begin{proof}  Using Theorem \ref{th:L^2Extension} for $q=0$, we have   that any holomorphic $(p,0)$-form on $\Omega^+$ 
extends to be a holomorphic
$(p,0)$ in $\mathbb{CP}^n$, which are zero (when $p>0$) or constants (when $p=0$). \end{proof}

\begin{theorem}\label{th:W^kExistence} Let $\Omega^+$ be a pseudoconcave domain 
in $\CP^n$ with Lipschitz boundary, where $n\ge 3$.
  For any $\dbar$-closed $f\in W^{k}_{p,q}({\Omega}^+)$, where $0\le p\le n$, $1\le q< n-1$, $p\neq q$ and $k\in \N$, 
   there exists
$u\in W^{k}_{p,q-1}(\Omega^+)$ with
$\dbar  u=f$ in $\Omega^+$.  
\end{theorem}

\begin{proof}  Let $F\in W^{k-1}_{p, q}(\mathbb{CP}^n)$ be the $\bar{\partial}$-closed extension of $f$ from $\Omega$ to $\mathbb{CP}^n$. Since
$$H^{p, q}_{W^{k-1}}(\mathbb{CP}^n)=\{0\},$$ there exists $u\in W^{k}_{p, q-1}(\Omega)$ such that $\bar{\partial} u=F$ on $\mathbb{CP}^n$. By the elliptic theory of the $\bar{\partial}$-complex on compact complex manifolds, one can choose such a solution  $u\in W^k_{p, q-1}(\mathbb{CP}^n)$. \end{proof}

 For $q=n-1$,  there is an additional compatibility
condition  for the $\dbar$-closed  extension of
$(p,n-1)$-forms from $\Omega^+$ to the whole space $\CP^n$.   This case differs from the others since the cohomology group does not vanish 
in general (see \cite{FLS17}). 
We first derive the compatibility condition for the extension of $\dbar$-closed forms when $q=n-1$.

  \begin{lemma}\label{le:closed compatibility} Let $\Omega $ be a
pseudoconvex domain in $\mathbb{CP}^n$ with Lipschitz
boundary and let $\Omega^+=\mathbb{CP}^n\setminus \overline\Omega$.
  For any $f\in W^{k}_{p,n-1}({\Omega}^+)$, $k\in \N$  and  $\phi\in L^2_{n-p,0} (\Omega,\delta^{2k-2})
\cap\text{Ker}(\dbar)$, the pairing 
\begin{equation}\label{eq:weakpair}
\int_{b\Omega^+}    f\wedge \phi        \end{equation}
is well-defined. 
\end{lemma}
\begin{proof}
 Since the boundary is Lipschitz,   any  function  in $W^k(\Omega^+)$   has a trace in $W^{k-\frac 12}(b\Omega^+)$.  Also  holomorphic   functions or 
 forms in $L^2(\Omega, \delta^{2k-2})$  have trace in $W^{-k+\frac 12}(b\Omega)$.   The pairing  \eqref{eq:weakpair} is well-defined follows from these known facts on Lipschtiz domains. The rest of the  proof of the lemma  is exactly the same as in \cite{Shaw10} and we give a sketch of the arguments.  
 
 Since the boundary is Lipschitz, for   any
$\dbar$-closed (holomorphic) 
$(n-p,0)$-form $\phi$ with $L^2(\Omega,\delta^{2k-2})$ coefficients,  there exists a sequence $\phi_\nu\in C^\infty_{n-p,0}(\overline \Omega)$ such that $\phi_\nu\to \phi$ and $\dbar \phi_\nu\to 0$   in $L^2(\Omega,\delta^{2k-2})$-norm.  This implies that $\phi_\nu\to \phi$ in   $W^{-k+1}(\Omega)$ norm since $\phi$ is holomorphic.

Let $\tilde f\in W^k_{p,n-1}(\mathbb{CP}^n)$ be a bounded extension of $f$. We have
\begin{equation}\label{eq:integration1}\int_{b\Omega} f\wedge \phi_\nu= \int_{\Omega} \dbar (\tilde f\wedge \phi_\nu) =\int \dbar \tilde f\wedge \phi_\nu \pm \int \tilde f\wedge \dbar\phi_\nu\to \int \dbar \tilde f \wedge \phi.\end{equation}
Thus the limit on the left-hand-side of \eqref{eq:integration1} exists and  is independent of the approximating  sequence $\{\phi_\nu\}$  that we choose. It is also independent of the extension function $\tilde f$.     Hence the pairing \eqref{eq:weakpair}
is well-defined. 
\end{proof}
 
\begin{theorem}\label{th:(n-1)-forms}  Let $\Omega $ be a
pseudoconvex domain in $\mathbb{CP}^n$ with Lipschitz
boundary and let $\Omega^+=\mathbb{CP}^n\setminus \overline\Omega$. 
  For  $\dbar$-closed  $f\in W^{k}_{p,n-1}({\Omega}^+)$, where $k\ge 1$,  $0\le p\le n$ and $p\neq n-1$, the following conditions are equivalent:
  
  \begin{enumerate}
  
    \item    The restriction of  $f$ to $b\Omega^+$  satisfies the
compatibility condition
\begin{equation}
\label{eq:compatibility}
\int_{b\Omega^+}    f\wedge \phi     =0,  \quad\phi\in L^2_{n-p,0} (\Omega,\delta^{2k-2})
\cap\text{Ker}(\dbar).  \end{equation}

  \item    There exists
$F\in W^{k-1}_{p,n-1}(\CP^n)$ such that
$F|_{\Omega}=f$ in $\Omega^+$  and  $\dbar F=0$ in $\CP^n$ in the
sense of distribution.

 \item   There exists $u\in W^k_{p,n-2}(\Omega^+)$ satisfying
  $\dbar u=f$ in $\Omega^+. $

\end{enumerate}

\end{theorem}

 \begin{corollary}\label{co:closedrange}   Let $\Omega^+$ be the same as in Theorem \ref{th:(n-1)-forms}. 
  Then  $\dbar : W^{k}_{p,n-2}({\Omega}^+)\to W^k_{p,n-1}(\Omega^+)$ has closed range,  where $k\ge 1$ and  $0\le p\le n$. 
 
\end{corollary}

\begin{proof}
 Let $f$ be a $\dbar$-closed $(p,n-1)$-form in $ W^k_{p, n-1}(\Omega^+)$. Suppose that $f$ is in the closure of the range of 
 $\dbar : W^{k}_{p,n-2}({\Omega}^+)\to W^k_{p,n-1}(\Omega^+)$. There exists a sequence $u_\nu \in W^k_{p,n-2}(\Omega^+)$ such 
 that $\dbar u_\nu\to f$ in $W^k_{p,n-1}(\Omega^+)$. It suffices to show  that there exists $u\in  W^k_{p,n-2}(\Omega^+)$ such that $\dbar u=f$. 
 
 From Theorem \ref{th:(n-1)-forms}, it suffices to show that the condition \eqref{eq:compatibility} is satisfied for every $\phi\in L^2_{n-p,0} (\Omega,\delta^{2k-2})
\cap\text{Ker}(\dbar)$. 
This follows from 
\begin{equation}\label{eq:IntegrationW1}
\int_{b\Omega^+}    f\wedge \phi     =\lim_{\nu\to\infty}\int_{b\Omega^+}  \dbar u_\nu\wedge \phi  = \lim_{\nu\to\infty} (-1)^{p+n-2}\int_{b\Omega^+} u_\nu
\wedge\dbar \phi=0.\end{equation}
Thus $f=\dbar u$ for some $u\in W^k_{p,n-2}(\Omega^+)$. Thus the range of $\dbar$ is closed in $W^k_{p,n-1}(\Omega^+)$.

\end{proof}

 Combining the results, we have proved the following theorem. 
 
\begin{theorem}\label{th:W^1cohomology}   Let $\Omega^+$ be the same as in Theorem \ref{th:(n-1)-forms}. Then for any $k\in \N$, 
\begin{itemize}
\item  $H^{p,q}_{W^k}(\Omega^+)=0$,  if $0\le q< n-1$  and  $p\neq q$;
\item 
$H^{p,n-1}_{W^k}(\Omega^+)$ is Hausdorff and infinite dimensional,  if  $p\neq n-1$.
\end{itemize}
\end{theorem}

\begin{remark}  It is still an open question if Theorems \ref{th:W^kExistence} and  \ref{th:(n-1)-forms} hold for $k=0$ (see Problem \ref{prob:L2 pseudoconcave}).  The missing ingredient is the lack of $W^1$-estimates with pseudoconvex domains in $\CP^n$.

 When the domain $\Omega$ is a  bounded domain with smooth boundary   in $\C^n$, there  has been a lot of  results obtained earlier.  
The space of $L^2$  harmonic forms  for the critical degree $q=n-1$  on an annulus between two concentric balls or  strongly pseudoconvex domains in $\C^n$ has been  computed in \cite{Hormander04}. This has been generalized to annulus between two pseudoconvex domains
in $\C^n$ in \cite{Shaw10, Shaw11}. We also remark that the conditions on the cohomology groups can be used to characterize domains with holes with Lipschtiz boundary in $\C^n$ (see     \cite{FLS17}). All these results depend on the Sobolev estimates for $\dbar$ proved by Kohn (see Theorem \ref{th:Kohn main}). 

 \end{remark}


  \section{Properties of holomorphic functions and forms on the Hartogs triangles} \label{sec:Hartogs 21}

We denote the homogeneous coordinates in $\CP^2$  by
$[Z_0,Z_1,Z_2]$. 
Let  $\hb^+$ and $\hb^-$  be the Hartogs triangles  defined by
\begin{align*}
\hb^+&=\{[Z_0:Z_1:Z_2]\in\cx\mathbb P^2~|~|Z_1|<|Z_2|\}\\
\hb^-&=\{[Z_0:Z_1:Z_2]\in\cx\mathbb P^2~|~|Z_1|>|Z_2|\}
\end{align*}
then $\hb^+\cap\hb^-=\emptyset$ and $\ol\hb^+\cup\ol\hb^-=\cx\mathbb P^2$.

 Let $U_j=\{[Z_0,Z_1,Z_2]\mid Z_j\neq 0\},$ $j=0,1,2.$ Then $\hb^+\subset U_2$. In local coordinates,   $$\hb^+= \{(z,w)\in \C^2\mid |w|<1\}.$$
 Thus $\hb^+$  is the  product  $\C\times D$, where $D$ is the unit disk. 
Hence $\hb^+$ is  pseudoconvex.

 In this section, we first recall some known results on the Hartogs triangles. 
 The Hartogs triangles are not Lipschitz. However,  
 some function properties  for the Hartogs triangles still hold.  Recall that   a domain $\Omega\subset \CP^n$ is called a Sobolev extension domain if for any 
$f\in W^s(\Omega)$, there exists $\tilde f \in W^s(\CP^n)$ such that $\tilde f=f$ on $\Omega$. 
 \begin{lemma}\label{le:extension hartogs cp2} The Hartogs triangles $\hb^+$ and $\hb^-$  are Sobolev extension domains.
 
 \end{lemma}

 \begin{proof}
 The Hartogs triangle $\hb^+$ is  smooth except at the point $[1,0,0]$.  If  we set
$z= {z_1}/{z_0}$ and $w= {z_2}/{z_0}$, then the domain $\hb^+$  is defined by the inhomolgeneous coordinates $(z,w)$ by 
$$\hb^+=\{z,w)\in \C^2\mid |z|<|w|\}.$$
The Hartogs triangle $\hb^+$ and $\hb^-$ are not Lipschitz at $(0,0)$.   At (0,0), the singularity of $\hb^+$ and $\hb^-$ are the same as the Hartogs triangle  
$$T=\{(z,w)\in  \C^2\mid |z|<|w|<1\}.$$ 
It is proved  in  \cite{BFLS22} that $T$ is an extension domain. Thus the lemma follows from the same proof. 
\end{proof}

It is also proved in Theorem 3.13 in   \cite{BFLS22} that the weak and strong extensions of $\dbar$ are the same.    
Define the $L^2$ Dolbeault cohomolgy group with respect to  $\dbar_c$  as follows:
 $$H^{p,q}_{\dbar_c, L^2}(\hb^-)=\frac {\text{Ker}(\dbar_c)}{\text{Range}(\dbar_c)}.$$

The following  lemma is proved in Proposition 6 in  \cite{ChakrabartiShaw12}.
  \begin{lemma}\label{le:holoL2}
Let $\hb^+\subset\cx\mathbb P^2$ be the Hartogs' triangle.
Then we have the following:
\begin{enumerate}
\item The Bergman space of $L^2$ holomorphic functions   $L^2(\hb^+)\cap\mathcal{O}(\hb^+)$
on the domain $\hb^+$ separates points in $\hb^+$.

\item There exist non-constant  functions in the space $W^1(\hb^+)\cap\mathcal{O}(\hb^+)$.
However, this space does not separate points in $\hb^+$ and  is not dense in the Bergman space $L^2(\hb^+)\cap \mathcal{O}(\hb^+)$.
\item  Let $f\in W^2(\hb^+)\cap \mathcal{O}(\hb^+)$ be a holomorphic function on $\hb^+$ which is in the
Sobolev space $ W^2(\hb^+)$. Then $f$ is a constant.
  \end{enumerate}
\end{lemma}

  \begin{lemma}\label{le:holo 20} The following results hold:
  \begin{enumerate}
   \item  $H^{2,0}_{L^2}(\hb^+)=0;$
  \item  $  H^{0,1}_{L^2}(\hb^+)=0;$  
  \item  $H^{2,1}_{L^2}(\hb^\pm)$  is infinite dimensional.

  \end{enumerate}
   \end{lemma}
   
  \begin{proof}

   Let  $z=z_0/z_2$ and $w=z_1/z_2$. Then   $\hb^-$ is biholomorphic to $\C\times D$. 
   Let $\phi= f dz \wedge dw,$ where $f$ is holomorphic in $\C\times D$. 
   Since $\phi$ is a $(2,0)$-form, its $L^2$-norm is metric independent. We can just use the Euclidean metric on $\C\times D$.
   If $f\in L^2(\C\times D)$, then $f$ is $L^2$ on the  leaf  $(\cdot, w)$ a.e., where $w\in D$.  Since $f$ is holomorphic, this implies that $f=0$ on $\C\times D$.  Thus $\phi=0$  on $\hb^+$.  This proves (1).

    (2)  is already proved in     Corollary \ref{co:L2 vanishing}.  The proof of (3) uses  Lemma \ref{le:holoL2}. It    follows from \cite{LaurentShaw18} combining with  the results    in \cite{BFLS22} (see also  \cite{Shaw24}).
   \end{proof}

\section{Non-closed range property  for $\dbar$ on Hartogs triangles  $\CP^2$}\label{sec:nonclosed proof}

We now state and prove the main result in this paper.

\begin{theorem}\label{th:Nonclosed 21}    $\dbar  :L^2_{2,0}(\hb^+)\to L^2_{2,1}(\hb^+)$ does not have closed range. 

        \end{theorem}

\begin{corollary}\label{co:non Hausdorff}  $H^{2,1}_{L^2}(\hb^+)$ is non-Hausdorff.   
\end{corollary}
\begin{proof} We will show that the  corollary follows easily from the theorem. 
It is well-known that     $\dbar  :L^2_{2,0}(\hb^+)\to L^2_{2,1}(\hb^+)$   has closed range if and only if 
 $H^{2,1}_{L^2}(\hb^+)$ is Hausdorff (see  e.g.,   Treves \cite{Treves67}).    
 \end{proof}

 It remains to prove Theorem \ref{th:Nonclosed 21}. To prove the theorems, we need two lemmas.

\begin{lemma}\label{le:L2 W1 H} The following are equivalent:
\begin{enumerate} 
\item  $\dbar  :L^2_{2,0}(\hb^+)\to L^2_{2,1}(\hb^+)$   has closed range. 

\item $\dbar_c: L^2_{0,1}(\hb^+)\to L^2_{0,2}(\hb^+)$ has closed range and the range is $L^2_{0,2}(\hb^+)$. 

\item 
  $H^{0,1}_{W^1}(\hb^-)=0.$

  \end{enumerate}
  \end{lemma}
  \begin{proof}

   It  follows from Lemma \ref{le:dbar star c} that  $\dbar_c: L^2_{0,1}(\hb^+)\to L^2_{0,2}(\hb^+)$ has closed range if and only if $\dbar  :L^2_{2,0}(\hb^+)\to L^2_{2,1}(\hb^+)$   has closed range. 
  Using (1) in  Lemma \ref{le:holo 20} and  $L^2$ Serre duality, we have 
  and 
  \begin{equation}\label{eq:dbar_c 02 hb}H^{0,2}_{\dbar_c, L^2}(\hb^+)\cong  H^{2,0}_{L^2}(\hb^+)=0.\end{equation}
  This proves that the range of $\dbar_c$  is $L^2_{0,2}(\hb^+)$. We have proved that (1) implies (2). Thus (1) and (2) are equivalent. 
  
  Next we  prove that (2) implies (3). 
   Since $\hb^-$ is an extension domain by Lemma \ref{le:extension hartogs cp2}, let $\tilde f$ be an extension of $f$ to $W^1_{0,1}(\CP^2)$.   Let $f_c=\dbar \tilde f$. Then $f_c\in L^2_{0,2}(\hb^+)$. Using  (2),  there exists 
$u_c\in L^2_{0,1}(\hb^+)$ such that $\dbar_c u_c =f_c$ in $\CP^2$. Letting $F= \tilde f-u_c$. Then $F\in L^2_{0,1}(\CP^2)$,  $\dbar F=0$ in $\CP^2$ and $F=f$ in $\Omega^+$. 
Any $\dbar$-closed  $f\in W^1_{0,1}(\Omega^+)$ extends to be an $L^2$ $\dbar$-closed form $F$  in $\CP^2$. So $F=\dbar U$ with $U$ in $W^1(\CP^2)$. Letting $u=U\mid_{\Omega^-}$, then $u\in W^1(\hb^-)$ and $\dbar u=f$. We have proved (3).

  Finally, we   prove that (3) implies (2).  Let $f\in L^2_{0,2}(\hb^+)$. Then there exists $V\in W^1_{0,1}(\CP^2)$ such that $\dbar V=f$ in $\CP^2$. We set $v=V|_{\hb^-}$.  Using (3), we there exists $u\in W^1(\hb^-)$ such that  $\dbar u=v$ on $\hb^-$.
  Let  $\tilde u\in W^1(\CP^2)$ be an extension of $u$.

  We define 
  $$u_c= V-\dbar \tilde u.$$
  Then $\dbar u_c =f$ in $\CP^2$ and $u_c=0$ on $\hb^-$. This proves (2). The lemma is proved. 
  \end{proof}  
  
   \begin{lemma}\label{le:W1 separate}   Suppose $H^{0,1}_{W^1}(\hb^+)=0$. Then  functions in $W^1(\hb^+)\cap \mathcal O(\hb^+)$ separate points in $\hb^+$.
 \end{lemma} 
 \begin{proof}  The proof is similar to the proof of Proposition 4.3 in \cite{HenkinIordan00}.  Let $a$ and $b$ be two distinct points in $\hb^+$ and let $c\in \hb^-$. There exists a Riemann surface $S$ of degree 2 passing through $a,\ b,\ c$. The Riemann surface is given by  $S=\{[Z]=[Z_1,Z_2,Z_3]\mid P(Z) =0\}$,  where $P(Z)$ is a   homogeneous polynomial of degree 2 in $[Z]$. We can choose $c$ such that $dP\neq 0$ on $S$. Let $\widetilde \hb^+$ be an open neighborhood of $\overline \hb^+$ and $c\notin \tilde \hb^+$.  
 The Riemann surface $S\cap \widetilde \hb^+$ is Stein and there exists a Stein neighborhood $U$ of $S \cap \widetilde \hb^+$ such that  $U$ is Stein (see \cite{GunningRossi}).  Let $h$ be a holomorphic function in $U$ such that $h(a)\neq h(b)$. 
 
 Let $\chi$ be a function in $C^\infty_0(U)$ such that $\chi=1$ in a neighborhood of $S\cap \widetilde \hb^+$. Consider the $(0,1)$-form
 $$\alpha =\frac {(\dbar \chi )h}{P}.$$
 Then $\alpha\in C^\infty_{0,1}(\overline \hb^+)$. Using the assumption that $H^{0,1}_{W^1}(\hb^+)=0$, there exists $u\in W^1(\hb^+)$ such that 
  $\dbar u= \alpha$ in $\hb^+$. Let 
  $$F=\chi h - P u.$$
  Then $F\in W^1(\hb^+)\cap \mathcal O(\hb^+)$. Furthermore,  $F(a)=h(a)\neq  h(b)=F(b).$  The lemma is proved. 
  \end{proof}

\subsection{Proof of Theorem \ref{th:Nonclosed 21}}
 
   Suppose that  $\dbar  :L^2_{2,0}(\hb^+)\to L^2_{2,1}(\hb^+)$   has closed range. Using Lemma \ref{le:L2 W1 H},  
   $$H^{0,1}_{W^1}(\hb^+)=0.$$
    It follows from Lemma \ref{le:W1 separate} that  holomorphic functions in $W^1(\hb^+)$ separate points. This is a contradiction
  to (2) in Lemma \ref{le:holoL2}.
  Theorem \ref{th:Nonclosed 21} is proved. \qed
  
  \bigskip

  We have also proved the following corollary.
  
  \begin{corollary}\label{co:H 21 W1}  $H^{0,1}_{W^1}(\hb^+)\neq 0.$
  \end{corollary}

 \section{Open problems}\label{sec:Open Problems} 
There are numerous open problems concerning $\dbar$ on domains in $\C^n$ and $\CP^n$. 
In this section we  list only a few   open  problems which are related to this article.

\begin{problem}[\textbf{Sobolev Estimates for $\dbar$ on Pseudoconvex domains  $\CP^n$}]\label{prob:Sobolev CPn}
Let $\Omega$ be a bounded pseudoconvex domain in $\CP^n$ with smooth boundary. Can one solve
 $\dbar$ with  $W^s$ estimates for all $s>0$? 
 In other words,  prove (or disprove) 
\begin{equation}\label{eq:H^s problem}H^{p,q}_{W^s}(\Omega)=0.\end{equation}
  
\end{problem}

We remark  that by Corollary \ref{co:H 21 W1}, some smoothness of $\Omega$ must be assumed. 
 When $\Omega$ is a bounded pseudoconvex  domain in $\C^n$ with smooth boundary, this is exactly Kohn's therorem (Theorem \ref{th:Kohn main}).

 \begin{problem}[\textbf{$L^2$ Existence  for $\dbar$ on Pseudoconcave domains  $\CP^n$}]\label{prob:L2 pseudoconcave}
Let $\Omega$ be a bounded pseudoconvex domain in $\CP^n$ with smooth (or Lipschitz) boundary and let $\Omega^+= \CP^n\setminus \overline \Omega$. Prove (or disprove) that  
 $$  H^{p,q}_{L^2}(\Omega^+)=0$$
 if $p\neq q$ and $q<n-1$. 
  
\end{problem}

 \begin{problem}\label{prop:Lip W1} Let $\Omega$ be a bounded pseudoconvex domain in $\C^n$ with Lipschitz boundary.  Determine if 
 $$H^{0,1}_{W^1}(\Omega)=0.$$
 \end{problem} 
 In other words, can one extend Kohn's results for $s=1$  (Theorem \ref{th:Kohn main})  to Lipschitz pseudoconvex domains?  This.question has been raised many years ago. It has been shown that $W^1$ estimates hold for pseudoconvex domains with $C^2$ boundary  (see \cite{Harrington07}). But it remains unsolved for general Lipschitz domains.  When
 the domain is the bidisk $D\times D$, this is proved rather recently (see \cite{CLS18}).  One can also ask similar questions for the Hartogs triangle in $\C^2$
 (see Problem \ref{prob:T Hausdorff} and Problem \ref{prob:T W1}).

\bigskip

Let $\Omega$ be a bounded   domain in $\CP^n$ with smooth  boundary. Let  $0\le p\le n$ and $1\le q\le n-1$. 
 Consider the induced operator $\dbar_b:L^2_{p,q-1}(b\Omega)\to L^2_{p,q}(b\Omega)$, the tangential Cauchy-Riemann operator on $b\Omega$. 
 
Let  $\vartheta_b : L^2_{p,q}(b\Omega)\to L^2_{p,q-1}(b\Omega)$ be the adjoint operatorr of $\dbar_b$ with respect to the Fubini-Study metric. Let
$$\square_b=\dbar_b\vartheta_b+\vartheta_b\dbar_b:L^2_{p,q}(b\Omega)\to L^2_{p,q}(b\Omega)$$
be the $\dbar_b$-Laplacian (or Kohn-Rossi operator).  
 
 \begin{problem}[\textbf{Kohn-Rossi Cohomology}]
Let $\Omega$ be a bounded pseudoconvex domain in $\CP^n$ with smooth  boundary.  Does $\dbar_b:L^2_{p,q-1}(b\Omega)\to L^2_{p,q}(b\Omega)$ have closed range   where $0\le p\le n$ and $1\le q\le n-1$?

If $\dbar_b$ has closed range for all degrees, show that for $1\le q<n-1$, 
  show that  the dimension of the Kohn-Rossi cohomology vanishes, i.e., 
  $$\dim_\C H^{p,q}_{\dbar_b, L^2}(b\Omega)=\dim_\C\text{Ker}(\square_b^{p,q})=0.$$ 
  If not, what is the dimension of the Kohn-Rossi cohomology $\dim_\C H^{p,q}_{\dbar_b, L^2}(b\Omega)$?

\end{problem}

We remark the following results are known: 
 If $\Omega$ is strongly pseudoconvex,  then we have that $\dbar_b$ has closed range following  the results of Kohn-Rossi \cite{KohnRossi65}.
 
   If $\Omega$ is a bounded pseudoconvex domain with smooth boundary in $\C^n$, the closed range property and $L^2$ existence  for $\dbar_b$ is proved in \cite{Shaw85Invent} for $q<n-1$ and in \cite{BoasShaw86,Kohn86} for $q=n-1$. 
In this case, we have that the Kohn-Rossi cohomology vanishes for $q<n-1$. 
 
 .

 \begin{problem}\label{prob:H 11 Hausdorff} 
 Determine if the $L^2$ Dolbeault cohomology on the Hartogs triangles for (1,1)-form  $H^{1,1}_{L^2}(\hb^\pm)$ satisfies
 
\begin{enumerate}
 \item $H^{1,1}_{L^2}(\hb^\pm)$ is Hausdorff,
 \item $H^{1,1}_{L^2}(\hb^\pm)=0.$
 \end{enumerate}

 \end{problem} 
 
 We  have proved that $H^{2,1}_{L^2}(\hb^\pm)$ is non-Hausdorff in Theorem \ref{th:Nonclosed 21}. We also know that $H^{0,1}_{L^2}(\hb^\pm)=0.$. It remains to investigate the closed range property for $(1,1)$-form.

\begin{problem}\label{prob:T Hausdorff} Let  $T$ be the Hartogs triangle in $\C^2$.

\begin{enumerate}
\item Let  $B$ be a ball of radius 2 centered at 0. Determine if $H^{0,1}(B\setminus \overline T)$ is Hausdorff. 
\item Determine the spectrum of the $\dbar$-Neumann operatoor on $T$,
\item Determine the spectrum of the $d$-Neumann operator,
\end{enumerate}
\end{problem}
This problem is raised in \cite{BFLS22}.   Since $T$ is pseudoconvex and bounded, H\"ormander's $L^2$ existence theorem  holds for $T$. We have 
 $$H^{p,1}_{L^2}(T)=0\qquad \text{for all }0\le p\le 2.$$
  
Notice that $H^{0,1}(B\setminus \overline T)$ is Hausdorff is equivalent to $H^{2,1}(B\setminus \overline T)$ is Hausdorff.    This question is equivalent to the following question (see  \cite{BFLS22} and  also \cite{LaurentShaw13}). 
 
\begin{problem}\label{prob:T W1} Determine if 
  $$H^{0,1}_{W^1}(T)=0.$$   
 \end{problem}
It  is proved in a recent paper \cite{PanZhang}  that 
 $$H^{0,1}_{W^{k,p}}(T)=0,\qquad k\in \N,  \ p>  4.$$
 However, it is still not known if this holds  for $W^{1,2}(T)=W^1(T).$
 By Corollary   \ref{co:H 21 W1}, we  have    $$H^{0,1}_{W^1}(\hb^+)\neq 0.$$

There are numerous other interesting problems on $\dbar$ which are yet to be understood. 
 We  list  only these few  problems  to highlight  the   importance of understanding the $L^2$-Sobolev theory for $\dbar$ on domains in  complex manifolds.
 
\bibliography{survey}
\providecommand{\bysame}{\leavevmode\hbox
	to3em{\hrulefill}\thinspace}

\end{document}